\documentclass[11pt]{article}

\usepackage{lineno}

\usepackage{amssymb,amsthm,amsmath}
\usepackage{amsfonts}
\usepackage{graphicx}
\usepackage[small]{caption}
\usepackage{subcaption}
\usepackage{epsfig}
\usepackage{pdfpages}
\usepackage{complexity}

\usepackage{psfrag}
\usepackage{epsfig}
\usepackage{epsf}
\usepackage{latexsym}
\usepackage{enumerate}
\usepackage{bbm}
\usepackage{url}
\usepackage{hyperref}

\newtheorem{theorem}{Theorem}
\newtheorem{lemma}[theorem]{Lemma}

\newtheorem{corollary}[theorem]{Corollary}
\newtheorem{conjecture}[theorem]{Conjecture}

\numberwithin{theorem}{section}

\newcommand{\ie}{{i.e.}}
\newcommand{\eg}{{e.g.}}

\newcommand{\cost}{{\rm cost}}

\newcommand{\dist}{{\rm dist}}
\newcommand{\diam}{{\rm diam}}

\newcommand{\NN}{\mathbb{N}} 
\newcommand{\RR}{\mathbb{R}} 
\newcommand{\eps}{\varepsilon}

\newcommand{\later}[1]{}
\newcommand{\old}[1]{}

\title{On a Traveling Salesman Problem for Points in the Unit Cube}

\author{%
J\'ozsef Balogh\footnote{Department of Mathematics, University of Illinois at Urbana-Champaign, Urbana, Illinois 61801, USA,
E-mail: \texttt{jobal@illinois.edu}. Research is partially supported by NSF grant DMS-1764123.}
\and Felix Christian Clemen\footnote{Corresponding author, Department of Mathematics, Karlsruhe Institute of Technology,
 76131 Karlsruhe, Germany, E-mail: \texttt{felix.clemen@kit.edu}}
\and Adrian Dumitrescu\footnote{Algoresearch L.L.C., Milwaukee, WI 53217, USA,
  E-mail: \texttt{ad.dumitrescu@algoresearch.org}}
}

\begin{document}

\maketitle

\begin{abstract}
  Let $X$ be an $n$-element point set in the $k$-dimensional unit cube $[0,1]^k$ where $k \geq 2$.
  According to an old result of Bollob{\'a}s and Meir (1992), there exists a cycle (tour) $x_1, x_2, \ldots, x_n$
  through the $n$ points, such that $\left(\sum_{i=1}^n |x_i - x_{i+1}|^k \right)^{1/k} \leq c_k$, where
  $|x-y|$ is the Euclidean distance between $x$ and $y$, and $c_k$ is an absolute constant that
  depends only on $k$, where $x_{n+1} \equiv x_1$.
  From the other direction, for every $k \geq 2$ and $n \geq 2$, there exist $n$ points in  $[0,1]^k$,
  such that their shortest tour satisfies $\left(\sum_{i=1}^n |x_i - x_{i+1}|^k \right)^{1/k} = 2^{1/k} \cdot \sqrt{k}$.
  For the plane, the best constant is $c_2=2$ and this is the only exact value known.  
  Bollob{\'a}s and Meir showed that one can take $c_k =  9 \left(\frac23 \right)^{1/k} \cdot \sqrt{k}$ for every $k \geq 3$
  and conjectured that the best constant is $c_k =  2^{1/k} \cdot \sqrt{k}$, for every $k \geq 2$. 
  Here we significantly improve the upper bound and show that one can take
  $c_k = 3 \sqrt5 \left(\frac23 \right)^{1/k} \cdot \sqrt{k}$ or $c_k = 2.91 \sqrt{k} \ (1+o_k(1))$. 
  Our bounds are constructive. We also show that $c_3 \geq 2^{7/6}$, which disproves the conjecture for $k=3$.

  Connections to matching problems, power assignment problems, related problems, including algorithms, are discussed
  in this context. A slightly revised version of the Bollob{\'a}s--Meir conjecture is proposed.


\end{abstract}

\section{Introduction} \label{sec:intro}

Given $n$ points in the unit square, Newman~\cite[Problem~57]{New82}
proved that there is a closed polygonal Hamiltonian cycle (tour) $H$ through the $n$ points
such that the sum of the squares of  its edge-lengths is at most $4$.
The upper bound of $4$ cannot be improved: Figure~\ref{fig:square}
shows three different point sets whose optimal tours yield exact equality.
More importantly, the above upper bound is independent of $n$. 
\begin{figure}[ht]
\centering   
\includegraphics[scale=0.9]{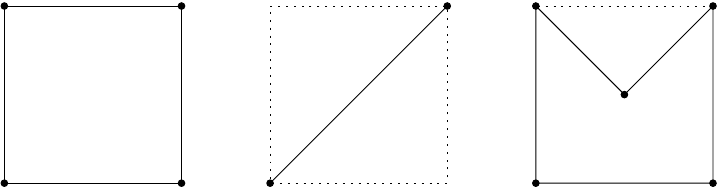}      
\caption{Tight examples with $4,2$, and $5$ points: $1+1+1+1= 2+2 = 1+1+1+ 1/2+1/2$.}
\label{fig:square}
\end{figure}
Meir~\cite{Mos91} considered the extension of this problem to higher dimensions.
For a point $x \in \RR^k$, let $|x|$ denote the Euclidean length of $x$; namely, if
$x=(\xi_1,\xi_2,\ldots,\xi_k)$, then 
$$|x|= \left( \sum_{i=1}^k \xi_i^2 \right)^{1/2}.$$ 
For two points $x,y \in \RR^k$, let the \emph{weight} of the edge $e=xy$, be $|e|:=|x-y|$, \ie, the Euclidean
distance between $x$ and $y$.

Let $X$ be an $n$-element point set in the unit cube $[0,1]^k$. For a graph $G$ on vertex set $X$, set 
\begin{align}
  S_k(G) &= \sum_{e\in G} |e|^k \quad \text{ and } \quad s_k(G) = \left(\sum_{e\in G} |e|^k\right)^{1/k}. \label{eq:S_k} 
  \end{align}
We refer to $S_k(G)$ and  $s_k(G)$ as the \emph{unscaled} and \emph{scaled} costs, respectively.
Denote by  $S_k^{\texttt{HC}}(X)$,  $S_k^{\texttt{ST}}(X)$ and $S_k^{\texttt{HP}}(X)$
($s_k^{\texttt{HC}}(X)$, $s_k^{\texttt{ST}}(X)$ and $s_k^{\texttt{HP}}(X)$)
the minimum over $S_k(G)$ ($s_k(G)$) where $G$ is a Hamiltonian cycle,
respectively a spanning tree or Hamiltonian Path with vertex set $X$. Further, let  
 \begin{gather*}
   s_k^{\texttt{HC}}(n) = \sup  \{ s_k^{\texttt{HC}}(X) \colon X\subseteq [0,1]^k, |X|=n \},
   \quad   s_k^{\texttt{ST}}(n) = \sup  \{ s_k^{\texttt{ST}}(X) \colon X\subseteq [0,1]^k, |X|=n \}, \\
     s_k^{\texttt{HP}}(n) = \sup  \{ s_k^{\texttt{HP}}(X) \colon X\subseteq [0,1]^k, |X|=n \}, \\
     s_k^{\texttt{HC}} = \sup_{n \geq 2} s_k^{\texttt{HC}}(n), \quad  s_k^{\texttt{ST}} =
     \sup_{n \geq 2}  s_k^{\texttt{ST}}(n) \quad \text{and} \quad s_k^{\texttt{HP}} = \sup_{n \geq 2}  s_k^{\texttt{HP}}(n).  
  \end{gather*}

It is clear that $s_k^{\texttt{HC}}(n) \geq s_k^{\texttt{HC}}(m)$, whenever $n \geq m$ (by clustering points and taking the limit). 
In this notation,  Newman's result mentioned earlier reads $s_2^{\texttt{HC}}(n) =2$ for every $n \geq 2$.
A~more recent reference to this result can be found in~\cite[Problem~124]{Bol06}.    
Currently this is the only exact value known.  
Meir~\cite{Mos91} asked whether $s_k(n)$ is bounded from above by a constant $c_k>0$ for every $k$.
Soon after, Bollob{\'a}s and Meir~\cite{BM92} answered Meir's question in the positive by proving that 
$ s_k^{\texttt{HC}}(n) \leq  9 \left(\frac23 \right)^{1/k} \cdot \sqrt{k}$ for every $k \geq 3$ and $n \geq 2$
(and recall that $c_2=2$). 
From the other direction, the $2$-point example consisting of two opposite vertices of $\{0,1\}^k$ shows that 
$s_k^{\texttt{HC}}(n) \geq 2^{1/k} \cdot \sqrt{k}$ for every $k \geq 2$ and $n \geq 2$; 
see Figure~\ref{fig:square}\,(center). We record their result below.

\begin{theorem}  \label{thm:bollobas-meir} {\rm (Bollob{\'a}s and Meir~\cite{BM92})}.
Let $k\geq 3$ and $n\geq 2$. Then, 
$$2^{1/k}\sqrt{k} \leq s_k^{\texttt{HC}}(n) \leq 3^{2-\frac{1}{k}}2^{1/k}\sqrt{k}.$$
\end{theorem}

In the conclusion of their paper~\cite{BM92}, the authors conjectured that 
$s_k^{\texttt{HC}}(n) = 2^{1/k} \cdot \sqrt{k}$ for every $k \geq 2$ and $n \geq 2$. 
Meir~\cite{Mos91} also asked for an algorithm that computes a tour whose cost is bounded by a constant
depending on $k$. As we will see in more detail in Section~\ref{sec:exact}, 
Bollob{\'a}s and Meir's proof implicitly gives a positive answer to this latter question.
Similarly, our new bounds in Theorem~\ref{thm:s} and Corollary~\ref{cor:matching} are constructive too. 

\paragraph{Background and related work.}
The \emph{traveling salesman problem} (TSP) is perhaps the most studied problem in the theory of combinatorial
optimization. Its approximability depends on the particular version of the problem.
Specifically, TSP with Euclidean distances admits a polynomial-time approximation scheme~\cite{Ar98,Mi99}. 
If the distances form a metric, then the problem is $\MaxSNP$-hard~\cite{PY93} and the best
approximation ratio known is essentially $3/2$~\cite{Chr76,KKG+21}. 

Estimating the length of a shortest tour of $n$ points in the unit square with respect to  Euclidean distances
has been studied as early as 1940s and 1950s by Fejes T\'oth~\cite{Fej40}, Few~\cite{Few55},
and Verblunsky~\cite{Ve51}, respectively.
Few~\cite{Few55} proved that the (Euclidean) length of a shortest cycle (tour)
through $n$ points in the unit square $[0,1]^2$ is at most $\sqrt{2n}+7/4$.
The same upper bound holds for the minimum spanning tree~\cite{Few55}.
Few's bound was rediscovered in 1983 by Supowit, Reingold, and Plaisted~\cite{SRP83}.
A~slightly better upper bound for the shortest cycle, $1.392 \sqrt{n} + 7/4$, has been derived by Karloff~\cite{Ka89},
who also emphasized the difficulty of the problem. 
The current best lower bound for the length of such a cycle is due to Fejes T\'oth~\cite{Fej40} and Few~\cite{Few55}:
it is $\left(\frac{4}{3}\right)^{1/4} \sqrt{n} - o(\sqrt{n})$, where $(4/3)^{1/4} = 1.075\ldots$.
For every dimension $k \geq 3$, Few showed that the maximum length of a shortest tour through $n$ points
in the unit cube is $\Theta(n^{1-1/k})$. Moran~\cite{Mor84} studied the length of the shortest traveling salesman tour
through a set of $n$ points of unit diameter in $\RR^k$. 

The length of a shortest tour through a random sample $\{ X_1,\ldots,X_n\}$ of $n$ points
in the unit cube $[0,1]^k$ was determined by Beardwood, Halton, and Hammersley. Let this 
length be denoted by $L(X_1,\ldots,X_n)$.
If $\{X_i\}$ is a sequence of independent random variables with the
uniform distribution on $[0,1]^k$, then there is a constant $\beta(k)>0$ such that
\[ {L(X_1,\ldots,X_n) / n^{1-1/k}} \to \beta(k) \]
with probability one~\cite{BHH59}. Later, Rhee~\cite{Rh92} proved that
${\beta(k) / \sqrt{k}} \to  {1 / \sqrt{2 \pi e}}$, see also~\cite{St97}.
The relevance of the cube diagonal, $\sqrt{k}$, in the above formulas, can be also observed
in our estimates for $s_k(n)$; see Theorem~\ref{thm:s}\,(ii) and Conjecture~\ref{conj:cycle}. 

Expressions for the cost of a Hamiltonian cycle of the kind in~\eqref{eq:S_k} have been considered
in the context of \emph{power assignment} problems in wireless networks.
Let $X$ be an $n$-element point set in the unit cube $[0,1]^k$ and $\alpha \geq 1$ be a real number.
For a Hamiltonian cycle $H$ as above, one is interested in minimizing a cost of the form
\begin{equation} \label{eq:alpha}
  \cost(H)= \sum_{i=1}^n |x_i - x_{i+1}|^\alpha.
\end{equation}
Such costs typically reflect the \emph{energy} costs along the edges that make the
cycle~\cite{BNS+10,KLS06} in wireless network transmission. 
An illustrative example is that of a \emph{virtual token} floating through the network,
where sensor nodes can attach or read data from the token before sending it to the next
node on the cycle. One can speak about finding a traveling salesman tour (TSP tour) of minimum
energy cost~\cite{FLLN11}.
The fact that $k$ is the smallest value of $\alpha$ for which the cost in~\eqref{eq:alpha} is bounded
from above by a constant (depending on $k$ but independent on $n$) should be noted~\cite{BM92,KLS06};
a fine grid section in the cube proves this point. 

As pointed out in several places in the literature~\cite{And01,BC00,BNS+10,FLLN11}, simply computing a
short (even optimal) tour for the underlying Euclidean instance does not work, \ie, does not provide
a good approximation with respect to the power costs in~\eqref{eq:alpha}. Funke, Laue, Lotker and Naujoks~\cite{FLLN11}
showed that the cost of an optimal tour for the Euclidean instance can be a factor of $\Omega(n)$ larger than that
of optimal tour for the power costs (a simple example can be constructed with equidistant points
on a line or on a circle of large radius).

In \cite{FLLN11} a recursive algorithm was also presented, that given $n$ points in $\RR^2$, it constructs a TSP tour for edge costs $|pq|^\alpha = |e|^\alpha$, 
whose cost is at most $2 \cdot 3^{\alpha-1}$ times that
of a minimum spanning tree (MST) of the point set. Since the cost of an MST does not exceed that of an
optimal Euclidean TSP tour, their algorithm is $2 \cdot 3^{\alpha-1}$-factor approximation for the TSP
with power costs as in~\eqref{eq:alpha}. The authors further show that the approach extends to $\RR^k$
with the same ratio:

\begin{theorem}   \label{thm:flln11}
 {\rm (Funke, Laue, Lotker, and Naujoks~\cite{FLLN11})}.
 There exists a $2 \cdot 3^{\alpha-1}$-approximation algorithm for the TSP in $\RR^k$ if the edge weights are
 Euclidean distances to the power $\alpha$. 
\end{theorem}

If for some $\tau >1$ distances of a TSP instance satisfy
\[ \dist(x,z) \leq \tau \left( \dist(x,y) + \dist(y,z) \right), \]
for any three vertices $x,y,z$, we say that they satisfy the \emph{relaxed triangle inequality},
see~\cite{And01,BC00,Mom15}.
It is important to note that the metric with Euclidean distances to the power $\alpha$ satisfies
the relaxed triangle inequality with $\tau = 2^{\alpha-1}$; see~\cite{BNS+10,FLLN11}.
For $\alpha=2$ (\ie, TSP with squared distances), Theorem~\ref{thm:flln11} yields a $6$-approximation.
De Berg, van Nijnatten, Sitters, Woeginger and Wolff~\cite{BNS+10} obtained a $5$-approximation.

\paragraph{Our results.}

The upper bound $s_k^{\texttt{HC}}(n) \leq 9 \left(\frac23 \right)^{1/k} \cdot \sqrt{k}$, where $k \geq 3$, has stood
unchanged for $30$ years~\cite{BM92}. Here we obtain several improvements.

\begin{theorem} \label{thm:s}
  The following bounds are in effect:
  \begin{itemize} \itemsep 1pt
  \item [{\rm (i)}] There exists a $4$-element point set in $[0,1]^3$ such that the cost of the shortest tour
    is at least $2^{7/6}=2.24\ldots$.   
Consequently, $s_3^{\texttt{HC}}(n) \geq 2^{7/6}=2.24\ldots$, for every $n \geq 4$. 
\item [{\rm (ii)}] Let $X$ be an $n$-element point set in the $k$-dimensional unit cube $[0,1]^k$, $k \geq 3$.
  Then there exists a tour $H=x_1, x_2, \ldots, x_n$ through the
  $n$ points, such that
  $\left(\sum_{i=1}^n |x_i - x_{i+1}|^k \right)^{1/k} \leq 3 \sqrt5 \left(\frac23 \right)^{1/k} \cdot \sqrt{k}$.
  Consequently,
  $s_k^{\texttt{HC}}(n) \leq 3 \sqrt5 \left(\frac23 \right)^{1/k} \cdot \sqrt{k} = 6.708\ldots \cdot \left(\frac23 \right)^{1/k} \cdot \sqrt{k}$.
\item [{\rm (iii)}]  $H$ can be computed in time proportional to that needed for computing a MST of the points,
  in particular, in subquadratic time. 
\end{itemize}
\end{theorem}

Several sharper bounds are obtained for sufficiently large $k$. 
We note that the conjectured optimal configuration consisting of a diameter pair of the cube
as well as the lower bound construction we will present for $k=3$ in Theorem~\ref{thm:s}~(i)
are subsets of $\{0,1\}^k$. This raises the natural question if one can determine the maximum of $s_k^{\texttt{HC}}(X)$
if the point set $X$ is in $\{0,1\}^k$. We answer this question.

\begin{theorem} \label{thm:vertices}
There exists an integer $k_0$ such that for all $k\geq k_0$ the following holds.
If $X$ is an arbitrary subset of vertices of $\{0,1\}^k$, then there exists a Hamiltonian cycle $H$ through $X$ such that
$s_k(H) \leq 2^{1/k}\sqrt{k}$.  
\end{theorem}
The ``sufficiently large''
requirement for Theorem~\ref{thm:vertices} is in fact quite modest. The threshold $k_0$ is below $30$. Note that the bound in Theorem~\ref{thm:vertices} is attained for $|X|=2$. 

\begin{theorem} \label{thm:asymp-st}
  For the family of minimum spanning trees, we have
  \[ s_k^{\texttt{ST}} \leq  \sqrt{k} \ (1+o_k(1)).  \]
  Apart from the error term, this bound is best possible.
\end{theorem}

By transforming a minimum spanning tree into a Hamiltonian cycle by using the method of Sekanina~\cite{Sek60}
and Bollob{\'a}s and Meir~\cite{BM92}, we obtain $s_k^{\texttt{HC}} \leq  3\sqrt{k} \ (1+o_k(1))$.  
A further refinement based on a two-phase algorithm and a new greedy algorithm that maintains a collection of spanning paths
allows us to obtain the following sharper bound.

\begin{theorem} \label{thm:asymp-hc+}
  For the family of Hamiltonian cycles, we have
\[ s_k^{\texttt{HC}} \leq  2.91 \sqrt{k} \ (1+o_k(1)).  \]
\end{theorem}

When the number of points $n$ is bounded by a constant (independent of $k$),
we can obtain a better asymptotic bound, close to the conjectured value $2^{1/k}\sqrt{k}$.

\begin{theorem} \label{thm:fixedn}
\label{asymp}
Let $n\geq 2$ be fixed.   For the family of Hamiltonian cycles, we have
\[ s_k^{\texttt{HC}}(n)= 2^{1/k}\sqrt{k} \ (1+o_k(1)).  \]
\end{theorem}
Note however, that in Theorem~\ref{asymp} we require $n$ to be constant; it does not imply
$s_k^{\texttt{HC}}= 2^{1/k}\sqrt{k} \ (1+o_k(1))$.

\medskip
The improved upper bounds in Theorem~\ref{thm:s} and~\ref{thm:asymp-hc+}, 
have implications for the existence of Hamiltonian paths  and perfect matchings
whose costs are bounded from above by constants depending on~$k$. 
These are discussed in Section~\ref{sec:remarks}.

\section{Hamiltonian cycles: exact upper and lower bounds} \label{sec:exact}

\subsection{An improved lower bound for $k=3$} \label{subsec:LB}

In this subsection we prove Theorem~\ref{thm:s}(i).
Consider the four-element point set
\[ X=\{(0,0,0),(0,1,1),(1,0,1),(1,1,0)\} \subset [0,1]^3. \]
$X$ is in fact a \emph{binary code} of length $3$ with \emph{minimum Hamming distance} $2$; see, \eg,
\cite[Ch.~5]{vL99}. As such, the corresponding Euclidean pairwise distances are at least $\sqrt2$.
Consequently, the unscaled cost of any TSP tour $H$ is at least $S_k(H)\geq 4 \cdot (\sqrt2)^3  = 11.31\ldots$.
On the other hand, the conjectured~\cite{BM92} optimal unscaled cost was $2 \cdot (\sqrt3)^3 =  10.39\ldots$.
\qed

\smallskip
It is possible that the new lower bound gives the right value of $s_3^{\texttt{HC}}(n)$ for $n \geq 4$, see
Conjecture~\ref{conj:cycle} in Section~\ref{sec:remarks}. 

\paragraph{Remark.} Interestingly enough, for $k=4$, there exist (at least) two different point sets,
one with $n=2$ and the other with $n=8$, whose shortest tours have the same cost $S_4^{\texttt{HC}}(X)$ as
the conjectured value, $S_4^{\texttt{HC}}(n) = 2 \cdot (\sqrt{4})^4 = 32$.
The former set consists of a pair of diagonally opposite vertices, say, $\{(0,0,0,0),(1,1,1,1)\}$.
This is in fact the point set that is behind the conjectured maximum cost for \emph{every} $k$.
The latter set is a binary code of length $4$ with minimum distance $2$; for example, one can take
the eight binary vectors with an even number of ones:
\begin{align*}
X = \{(0,0,0,0),(0,0,1,1),(0,1,0,1),(0,1,1,0),(1,0,0,1),  (1,0,1,0), (1,1,0,0), (1,1,1,1)\}.
\end{align*}
Then $S_4^{\texttt{HC}}(X) \geq 8 (\sqrt2)^4 = 32$ and this value can be attained; equivalently, $s_4^{\texttt{HC}}(X) \geq 2^{5/4}$.  
We were not able to find two different sets $X$ with $s_k^{\texttt{HC}}(X) \geq 2^{1/k} \cdot \sqrt{k}$ for any other $k \geq 5$.

\subsection{An improved upper bound for every $k \geq 3$} \label{subsec:UB}

In this section we prove the last two items in Theorem~\ref{thm:s}.
Our proof is modeled by that in~\cite{BM92}. It uses a ball packing argument based on the following lemma.
(A similar lemma, however, with smaller ball radii, can be found in~\cite{KLS06}.)

\begin{lemma}   \label{lem:bm92}
 {\rm (Bollob{\'a}s and Meir~\cite{BM92})}.
 Let $T=(V,E)$ be a minimum spanning tree for a finite point set $X \subset \RR^k$. For each edge
 $e=xy \in E$ let $B_{e}$ be the open ball of radius $\frac14 |x-y|$ centered at $\frac12 (x+y)$. Then
 $B_{e} \cap B_{e'} =\emptyset$ whenever $e$ and $e'$ are edges of $T$. The factor $\frac14$
 is as large as possible.
\end{lemma}

In addition, a suitable order of traversing the vertices of a minimum spanning tree first developed by
Sekanina~\cite{Sek60,Sek71} is needed. The algorithm can be made to run in linear time.
A proof of this traversal result --- in slightly different terms --- also appears in~\cite{BM92}.
A~few definitions and notations (from~\cite{BM92}) are as follows. 
The \emph{$h$'th power} $G^h$ of a graph $G=(V,E)$ is the graph with vertex set $V$
and edge set $E(G^h) = \{xy \colon x,y \in V, 1 \leq d(x,y) \leq h\}$. Here $d(x,y)$ is the
distance between $x$ and $y$ in the graph.  Let $T$ be a tree and $xy \in E(T^h)$.
An edge $uv \in E(T)$ is said to be \emph{used by} $xy$ if the edge $uv$
  is on the unique path in $T$ (of length at most $h$) from $x$ to $y$. If $H$ is a subgraph of $T^h$,
  then an edge of $T$ is \emph{used $t$ times by} $H$ if it is used by $t$ edges of $H$. 

  \begin{lemma}   \label{lem:sek}
  {\rm (Sekanina~\cite{Sek60}, Bollob{\'a}s and Meir~\cite{BM92})}.
  Let $x$ be a vertex of a tree $T$ with at least $3$ vertices. Then $T^3$, the cube of $T$, contains a
  Hamiltonian cycle $H$ such that every edge of $T$ is used exactly twice by $H$, and one of the edges of
  $H$ incident to $x$ is an edge of $T$.
\end{lemma}

It implies the following lemma which is not stated explicitly in \cite{BM92} but is used in the proof
of their Theorem 3. For completeness, we include their proof here.

\begin{lemma}   \label{lem:bm92b} {\rm (Bollob{\'a}s and Meir~\cite{BM92})}.
  Let $T$ be a spanning tree for a finite point set $X \subset \RR^k$. Then there exists a Hamiltonian cycle $H$
  on $X$ such that
 \begin{align*}
 S_k(H)\leq  \frac{2}{3}\cdot 3^k \cdot S_k(T).
 \end{align*}
\end{lemma}
\begin{proof}   
Let $e_1,\ldots,e_n$ be the edges of a Hamiltonian cycle $H$ in $T^3$ guaranteed by Lemma~\ref{lem:sek}. 
Suppose that the edges of $T$ used by $e_i$ have lengths $d_{i_{1}},\ldots, d_{i_{\ell}}$, where $\ell \leq 3$. Set 
$f_i = d_{i_{1}}+\ldots+d_{i_{\ell}}$ and $f=(f_i)_{i\in [n]}\in \mathbb{R}^n$. Then $|e_i|\leq f_i$ for every $i$, each $f_i$ is a 
sum of at most three $d_j$'s and each $d_j$ occurs in 
the representations of two $f_i$'s. 

Now, we can form three vectors $v_1,v_2,v_3\in \mathbb{R}^n$
such that $f=v_1+v_2+v_3$, every coordinate of $v_i$ is a 
$d_j$ or 0, and every $d_j$ occurs exactly twice as a 
coordinate in the three $v_i$'s. Therefore, $\sum_{i=1}^3 \lVert v_i \rVert_k^k =  2\sum_{j=1}^{n-1}d_j^k$.
Hence, by the triangle-inequality and Jensen's inequality, 
\begin{align*}
  \lVert f \rVert_k &=\lVert v_1+v_2+v_3 \rVert_k \leq \sum_{i=1}^3 \lVert v_i \rVert_k\leq 3 \left(\frac{1}{3} \sum_{i=1}^3 \lVert v_i \rVert_k^k \right)^{1/k} \\
  &= 3 \left(\frac{2}{3} \sum_{j=1}^{n-1}d_j^k \right)^{1/k}=  3 \left(\frac{2}{3}\right)^{1/k} \cdot s_k(T),
\end{align*}
and thus
\begin{equation*}
S_k(H)=\sum_{i=1}^n |e_i|^k\leq \lVert f \rVert_k^k\leq \frac{2}{3}\cdot 3^k \cdot S_k(T).
\qedhere
\end{equation*}
\end{proof}

For convenience, here we work with the unit cube $U=[-1/2,1/2]^k$ centered at the origin $o=(0,\ldots,0)$.
Assume that $n \geq 3$, since it is clear otherwise that $s_k(H) \leq 2^{1/k} \cdot \sqrt{k}$.
 It was shown in~\cite{BM92}
that $\cup_{e\in T} B_e $ is contained in the ball of radius $0.75 \sqrt{k}$ centered at
the origin $o$. We next show that $\cup_{e\in T} B_e$ is contained in the ball of radius
$ \frac{\sqrt5}{4} \sqrt{k} =0.559 \ldots \cdot \sqrt{k}$ centered at $o$. The idea for the improvement
is that centers of balls corresponding to long edges of $T$ cannot be too far from the center of the cube.
The key step is the following.  

\begin{lemma} \label{lem:0.56}
Let $U=[-1/2,1/2]^k$ and $u,v \in U$. Then
\begin{equation} \label{eq:0.56}
  \frac{|u+v|}{2} + \frac{|u-v|}{4} \leq \frac{\sqrt5}{4} \sqrt{k}.
\end{equation}
 This inequality is the best possible.      
\end{lemma}
\begin{proof}
To start with, note that
\begin{align*}
|u|^2 \leq \sum_1 ^k \frac14 = \frac{k}{4}, \quad \quad
|v|^2 \leq \sum_1 ^k \frac14 = \frac{k}{4} \quad  \quad \text{and} \quad \quad
|u -v| \leq \sqrt{k}. 
\end{align*}
The first two relations immediately yield
\begin{equation} \label{eq:sq}
|u|^2 + |v|^2 \leq \frac{k}{2}. 
\end{equation}
Recall the identities
\begin{align} \label{eq:start}
|u+v|^2 = |u|^2 + |v|^2 + 2 uv, \quad  \quad |u-v|^2 &= |u|^2 + |v|^2 - 2 uv. 
\end{align}
Here $uv$ is the dot product of $u$ and $v$. We deduce that 
\[ |u-v|^2 = 2 (|u|^2+|v|^2) - |u+v|^2 \leq 2 (|u|^2+|v|^2). \]
We can thus write 
$|u-v| = \lambda \sqrt{|u|^2+|v|^2}$, where $0 \leq \lambda \leq \sqrt2$, whence
\[ |u-v|^2 = \lambda^2 (|u|^2+|v|^2). \]
From the two equations in~\eqref{eq:start} we also obtain
\begin{align*} 
|u+v| =  \sqrt{2(|u|^2 + |v|^2)  - |u-v|^2} = \sqrt{(2 - \lambda^2) (|u|^2 + |v|^2)}. 
\end{align*}
Substituting the expressions of $|u+v|$ and $|u-v|$ and using~\eqref{eq:sq} yields
\begin{align*}
  \frac{|u+v|}{2} + \frac{|u-v|}{4} &=
  \frac {\sqrt{(2 - \lambda^2) (|u|^2 + |v|^2)}}{2} + \frac{\lambda \sqrt{|u|^2+|v|^2}}{4} =
  \left( \frac{ \sqrt{2 - \lambda^2}}{2} + \frac{\lambda}{4} \right) \sqrt{|u|^2+|v|^2} \\
 &\leq \frac14 \left( \lambda + 2 \sqrt{2 -\lambda^2} \right) \sqrt{ \frac{k}{2}}. 
\end{align*}

A standard calculation shows that the function $f(\lambda) = \lambda + 2 \sqrt{2 -\lambda^2}$,
where $0 \leq \lambda \leq \sqrt2$,
attains its maximum, $\sqrt{10}$, at $\lambda= \sqrt{\frac25}$. Consequently,
\begin{align*}
  \frac{|u+v|}{2} + \frac{|u-v|}{4} &\leq
  \frac14 \sqrt{10} \sqrt{\frac{k}{2}} = \frac{\sqrt{5}}{4} \sqrt{k}.
\end{align*}
This concludes the proof of the upper bound.

\medskip
For a tight example, assume that $k$ is a multiple of $5$ and let $u=u_1,\ldots,u_k$, and
$v=v_1,\ldots,v_k$, where
\begin{align*}
 u_i &=
\left\{ \begin{array}{ll}
  +\frac12, & \text{ for } i=1,\ldots,\frac{4k}{5}, \\
  -\frac12, & \text{ for } i=\frac{4k}{5}+1,\ldots,k.
\end{array} \right.\\
v_i &=  +\frac12, \ \ \ \ \text{ for } i=1,\ldots,k.
\end{align*}

It is now easily verified that
\begin{align*}
\frac{|u+v|}{2} = \sqrt{\frac{4k}{5} \cdot \frac14}, \quad  \quad 
\frac{|u-v|}{4}  = \sqrt{\frac{k}{5} \cdot \frac{1}{16}}, \quad  \quad \text{and} \quad  \quad 
\frac{|u+v|}{2} + \frac{|u-v|}{4} &= \frac54 \cdot \sqrt{\frac{k}{5}} = \frac{\sqrt{5}}{4} \sqrt{k},
\end{align*}
as required. 
\end{proof}

\paragraph{Final argument in the proof of Theorem~\ref{thm:s}.}
Let $u,v \in U$ such that $e=uv$ is an edge of the MST $T$. By the triangle inequality, the distance from
the center of the cube to any point in the ball $B_{e}$ is at most $\frac12 |u+v| + \frac14 |u-v|$.
By Lemma~\ref{lem:0.56} this distance is at most $\frac{\sqrt5}{4} \sqrt{k}$, thus
$\cup_{e\in T} B_e  \subset B$, where $B$ is the ball of radius
$ \frac{\sqrt5}{4} \sqrt{k} =0.559 \ldots \cdot \sqrt{k}$ centered at $o$.

The ball packing argument in~\cite{BM92} yields
$S_k(T) \leq (3 \sqrt{k})^k$.
Using Lemma~\ref{lem:0.56} instead improves this bound to
$S_k(T) \leq (\sqrt{5k})^k$. By Lemma~\ref{lem:bm92b} we obtain a Hamiltonian cycle $H$ through $P$ satisfying

\begin{equation} \label{eq:new}
S_k(H) \leq \frac{2}{3}\cdot3^k \cdot S_k(T) \leq \frac{2}{3} \cdot3^k \cdot (5k)^{k/2}.
\end{equation}
Taking the $k$-th root completes the proof of item (ii). Note that the only change in the calculation is
replacing a multiplicative factor of $3$ by $\sqrt{5}$ (in~Inequality~(2) from~\cite{BM92}).
The improvement carries on proportionally and is reflected in the final bound. 

Recall that the traversal of the MST $T$ using the algorithm of Sekanina~\cite{Sek60,Sek71} takes linear time. As such,
the running time for computing the TSP tour is determined by the time to compute $T$. This proves
item (iii) and completes the proof of Theorem~\ref{thm:s}. 
\qed

\smallskip
An alternative way to verify the upper bound in~\eqref{eq:new} is by using Theorem~\ref{thm:flln11}.
The details are left to the reader.

\section{Hamiltonian cycles for subsets of cube vertices}

In this section we consider our problem (the study of extremal values for Hamiltonian cycles and paths in $[0,1]^k$)
when the input is restricted to subsets of cube vertices. Note that this restriction is quite natural, since all known best
constructions are attained or matched by such subsets. We will use some results on binary codes.

\subsection{Preparation: binary codes}
First we prove an optimization result which will be used multiple times throughout this  paper.

\begin{lemma}
  \label{lem: opt1}
Let $q_1,q_2,\ldots, q_m\in [0,1]$. Then,  
\begin{align*}
  \sum_{i<j} |q_i-q_j|^2 \leq
  \left \lfloor \frac{m}{2} \right \rfloor  \cdot  \left \lceil \frac{m}{2} \right \rceil. 
  \end{align*}
\end{lemma}
\begin{proof}
  We prove this result by induction on $m$. The statement holds trivially for $m=1$ and $m=2$.
  Let $q_1,q_2,\ldots, q_m\in [0,1]$ for some $m\geq 3$. We can assume $0=q_1\leq q_2\leq \ldots \leq q_m=1$. 
By the induction assumption, 
\begin{align*}
  \sum_{1<i<j<m} |q_i-q_j|^2 \leq
  \left \lfloor \frac{m-2}{2} \right \rfloor  \cdot  \left \lceil \frac{m-2}{2} \right \rceil. 
  \end{align*}
Observe that the maximum of the quadratic function $f(x) = x^2 + (1-x)^2$ over the interval $[0,1]$ is obtained
at $x=0$ or $x=1$. Thus, $|q_1-q_j|^2+|q_m-q_j|^2=q_j^2+(1-q_j)^2\leq 1$ for $j\in\{2,\ldots,m-1\}$. Therefore,
\begin{align*}
  \sum_{i<j} |q_i-q_j|^2&= |q_1-q_m|^2+\sum_{1<j<m} (|q_1-q_j|^2+|q_m-q_j|^2) + \sum_{1<i<j<m} |q_i-q_j|^2 \\
  &\leq 1+(m-2)+  \left \lfloor \frac{m-2}{2} \right \rfloor  \cdot  \left \lceil \frac{m-2}{2} \right \rceil=
  \left \lfloor \frac{m}{2} \right \rfloor  \cdot  \left \lceil \frac{m}{2} \right \rceil,
  \end{align*}
completing the proof of this lemma.  
\end{proof}

\begin{lemma} \label{lem:3incube}
  Let $\delta,\gamma> 0$, and $k_1,k_2$ be non-negative integers.
  Let $X\subseteq [0,\delta]^{k_1} \times [0,\gamma]^{k_2}$ be a finite set of size $|X|\geq m\geq 2$.
  Then there exists two distinct points $p,q\in X$ such that  
\[ |p-q|^2\leq \frac{\lfloor \frac{m}{2} \rfloor \lceil \frac{m}{2} \rceil}{\binom{m}{2}} \, (\delta^2k_1+\gamma^2k_2). \]
\end{lemma}
\begin{proof}
  Let $p_1,p_2,\ldots,p_m$ be any $m$ points from $X$. Given integers $i$ and $j$, we denote by ${p_i}_j$ the $j$-th coordinate of $p_i$. By applying Lemma~\ref{lem: opt1} and scaling we obtain
  \begin{align}
  \label{opt1}
  \sum_{i<i'} |p_{{i}_j} -p_{{i'}_j}|^2 &\leq
  \left \lfloor \frac{m}{2} \right \rfloor  \cdot  \left \lceil \frac{m}{2} \right \rceil \cdot \delta^2 \quad \quad
  \text{for every $j\in [k_1]$, and}\\
    \label{opt2}
  \sum_{i<i'} |p_{{i}_j} -p_{{i'}_j}|^2 &\leq
  \left \lfloor \frac{m}{2} \right \rfloor  \cdot  \left \lceil \frac{m}{2} \right \rceil \cdot \gamma^2 \quad \quad
  \text{for every $j\in [k_1+k_2]\setminus[k_1]$.}
    \end{align}
By summing up the inequalities \eqref{opt1} and \eqref{opt2}, we obtain
  \begin{align*}
    \sum_{i<i'} |p_{{i}} -p_{{i'}}|^2 \leq \left
    \lfloor \frac{m}{2} \right \rfloor  \cdot  \left \lceil \frac{m}{2} \right \rceil \cdot (\delta^2k_1+\gamma^2k_2). 
  \end{align*}
Thus, by averaging over all pairs of points, the minimizing pair satisfies the claimed inequality.
\end{proof}

Applying Lemma~\ref{lem:3incube} with $\delta=\gamma=1$, $k_1=k$ and $k_2=0$, immediately yields
the following symmetric version.

\begin{lemma} \label{lem:sym}
  Let $X\subseteq [0,1]^{k}$ of size $|X|\geq m \geq 3$.
Then there exist two distinct points $p,q \in X$ such that
  \[ |p-q|^2 \leq \frac{\lfloor \frac{m}{2} \rfloor  \cdot  \lceil \frac{m}{2} \rceil}{{m \choose 2}} \cdot k. \]
\end{lemma}

Let $A(k,d)$ denote the maximum cardinality of a binary code of length $k$ with minimum distance~$d$.
We recall the following fact~\cite{vL95}:

\begin{lemma} {\rm (Singleton bound).} \label{lem:singleton}
$A(k,d) \leq 2^{k-d+1}$.
\end{lemma}

We need the following improvement.

\begin{lemma}  \label{lem:singleton2}
If $d<\frac{2}{3}k$, then
$A(k,d) \leq  2^{k-\frac{3}{2}d+2}$.
\end{lemma}
\begin{proof}
Towards contradiction, assume that there exists $X\subseteq\{0,1\}^k$ of
size $|X|>2\cdot 2^{k-\frac{3}{2}d+1}$ such that $|p-q|^2\geq d$ for
every $p,q\in X$. By the pigeonhole principle, there exists $p,q,r \in X$
which coincide on the first $\lfloor k-\frac{3}{2}d+1 \rfloor$ coordinates.
By Lemma~\ref{lem:sym}, applied with $m=3$ to the last $\lceil \frac{3}{2}d\rceil-1$
coordinates, we get that
\begin{align*}
\min\{|p-q|^2,|p-r|^2,|r-q|^2\} \leq \frac{2}{3} \left(\left\lceil \frac{3}{2}d\right\rceil-1\right) < d,  
\end{align*}
a contradiction.
\end{proof}

\subsection{Building a path greedily}   \label{pathgreedy}

In the proofs of some of our results we will analyze a greedy algorithm which takes a discrete 
point set $X\subseteq [0,1]^k$ of size $|X|=n$ as an input and creates a Hamiltonian path $F$ through $X$.
It   processes the point pairs in nondecreasing order of distance and maintains a collection of paths.

\medskip
\noindent \textbf{Algorithm 1:} Initially, set 
$F_0$ to be the empty graph on $X$. For $i\in [n-1]$, let $e_i$ be an edge of smallest weight
among all edges $e\not\in F_{i-1}$ which satisfy that $F_{i-1}+e$ is a vertex-disjoint union of paths.
Set $F_{i}:=F_{i-1}+e_i$. Then, $F:=F_{n-1}$ is a Hamiltonian path.

\begin{lemma}
\label{bincode}
Let $j\in [k]$. The number of edges $e\in F$ satisfying $|e|^2\geq j$ is less than $A(k,j)$.
\end{lemma}
\begin{proof}
  Let $\ell$ be the smallest integer such  that $|e_\ell|^2\geq j$. The number of edges $e\in F$ satisfying $|e|^2\geq j$
  is less than the number of components in $F_{\ell}$, which is $n-\ell$. Let $P_\ell\subseteq X$ be a set containing one endpoint
  of each path in $F_{\ell}$. The set $P_\ell$ is a binary code of length $k$ with  minimum distance $j$. Thus, the number
  of edges $e\in F$ satisfying $|e|^2\geq j$ is less than $A(k,j)$.
\end{proof}

\paragraph{Proof of Theorem~\ref{thm:vertices}.}
If $|X|=2$, the statement holds trivially. Assume $n:=|X|\geq 3$. Let $F$ be the Hamiltonian path created by Algorithm 1.
We partition the edges $e\in F$ into four classes.
 
  \begin{enumerate} \itemsep 1pt
    \item \emph{short} edges: $|e|^2 \leq \frac{k}{5}$.
    \item \emph{medium} edges:  $\frac{k}{5} < |e|^2 \leq \frac{3k}{5}$.
  \item \emph{long} edges: $\frac{3k}{5} < |e|^2 \leq
    \frac{2k}{3} $.
  \item \emph{very long} edges:  $\frac{2k}{3} <|e|^2 $.
  \end{enumerate}
  Denote by $F^s,F^m,F^l,f^{vl}$ the subgraphs of $F$ containing all short, medium, long and very long edges,
  respectively.  They partition $F$ and thus  $S_k(F)= S_k(F^s)+S_k(F^m)+S_k(F^l)+S_k(F^{vl})$.
  We will provide upper bounds for the four contributions separately.  

  Since $n\leq 2^k$, the number of short edges is trivially at most $2^k$.
  Thus, 
  $$S_k(F^s)\leq 2^k \left(\sqrt{\frac{k}{5}}\right)^k.$$
Now, we estimate $S_k(F^m)$. Let $j$ be an integer satisfying $\frac{k}{5} < j \leq \frac{3k}{5}$.
The number of edges $e\in F$ satisfying $|e|^2\geq j$ is less than $ A(k,j)\leq 2^{k-\frac{3}{2}j+2}$
by Lemmas~\ref{lem:singleton2} and ~\ref{bincode}. Therefore,
  \[ S_k(F^m)\leq \sum_{j=\left\lceil\frac{k}{5}\right\rceil}^{\left\lfloor\frac{3k}{5}\right\rfloor} 2^{k-\frac{3}{2}j+2} \left(\sqrt{j}\right)^k
  \leq 4\cdot \left(\left\lfloor\frac{3k}{5}\right\rfloor-\left\lceil\frac{k}{5}\right\rceil\right)
  \cdot \left(0.842\sqrt{k}\right)^k
  \leq \frac{8k}{5}  \cdot \left(0.842\sqrt{k}\right)^k. \]
  Here we used that the function $f(x)=2^{1-3x/2} \sqrt{x}$, where $x\geq 0$,
  is maximized for $x=\frac{1}{\log(8)}$ and thus $2^{1-3x/2} \sqrt{x}\leq 0.842$. 

  Next, we estimate $S_k(F^l)$. The number of edges $e\in F$ satisfying $|e|^2>\frac{3k}{5}$
  is less than \quad $A(k,\lfloor \frac{3k}{5}\rfloor+1)\leq  4$ by Lemma~\ref{lem:sym},
    applied with $m=5$ and by Lemma~\ref{bincode}. Therefore,  
      \[S_k(F^l)\leq 3 \cdot \left(\sqrt{\frac{2}{3}k}\right)^k. \]

  Last, we estimate $S_k(F^{vl})$. The number of edges $e\in F$ satisfying $|e|^2>\frac{2k}{3}$
  is less than \quad $A(k,\lfloor \frac{2k}{3}\rfloor+1)\leq  2$ by Lemma~\ref{lem:sym},
applied with $m=3$ and Lemma~\ref{bincode}. Thus, there is at most one very long edge $e$ in $F$.
This very long edge has length at most $|e| \leq \sqrt{k-1}$ by the following argument.
Consider the last step of the greedy algorithm, when the last two paths, call them $P_1$ and $P_2$, are
 being joined. Since $|X|\geq 3$, one of them, say $P_1$, contains at least two
 vertices. An endpoint of the path $P_2$ has distance at most
 $\sqrt{k-1}$ to one of the endpoints of $P_1$, since not both
 endpoints can be opposite on the cube.
 Thus, $|e| \leq \sqrt{k-1}$.  We get $S_k(F^{vl})\leq \sqrt{k-1}^k$.
 
  Adding up the four contributions to $S_k(F)$ yields
  \begin{align}
  \label{ineq:4con}
  \nonumber
    S_k(F) &= S_k(F^s)+S_k(F^m)+S_k(F^l)+S_k(F^{vl})\\
     \nonumber
    &\leq 2^k \left( \sqrt{\frac{k}{5}} \right)^k
    + \frac{8k}{5}  \cdot \left(0.842\sqrt{k}\right)^k
    +3 \cdot \left(\sqrt{\frac{2}{3}k}\right)^k+\left(\sqrt{k-1}\right)^k\\
    &= \left(\sqrt{k}\right)^k\left(\left(\frac{2}{\sqrt{5}}\right)^k+\frac{8k}{5}\cdot
    0.842^k +3 \cdot \left(\sqrt{\frac{2}{3}}\right)^k +
    \left(\sqrt{\frac{k-1}{k}}\right)^k \right) \nonumber \\ 
    &< \left(\sqrt{k}\right)^k,
  \end{align}
  where the last inequality holds for $k$ sufficiently large. We used the fact that
  $\left(\sqrt{\frac{k-1}{k}}\right)^k$ converges to $e^{-1/2}$.
  Let $H$ be the Hamiltonian cycle obtained from $F$ by connecting the two endpoints.
  Then
 \[ \pushQED{\qed} 
  S_k(H) \leq S_k(F)+ \left(\sqrt{k}\right)^k \leq 2\left(\sqrt{k}\right)^k. \qedhere \popQED \]

We remark that the proof of Theorem~\ref{thm:vertices} works for $k_0=29$. The last inequality in \eqref{ineq:4con} is strict.
Thus, Theorem~\ref{thm:vertices} is tight only for $|X|=2$.

\section{Hamiltonian cycles: asymptotic upper bounds} \label{sec:asymptotic}

In this section we prove Theorems~\ref{thm:asymp-st}, \ref{thm:asymp-hc+} and \ref{thm:fixedn}.

\subsection{Preparation}

\begin{lemma} \label{lem:hypercube1}
  Let $0<\alpha < 1$ and $Y\subseteq [0,1]^k$ such that  $|u-v|> \alpha \sqrt{k}$ for every two distinct points $u,v\in Y$.
  Let $m\in  \NN$. Then,
$$|Y|\leq 2m \cdot \left\lceil  \sqrt{\frac{1}{2}\left(1+\frac{1}{2m-1}\right)}\alpha^{-1}\right\rceil^k.$$
\end{lemma}
\begin{proof}
  Let $\beta=\left\lceil  \sqrt{\frac{1}{2}\left(1+\frac{1}{2m-1}\right)}\alpha^{-1}\right\rceil$.
  Assume that $|Y|>2m\cdot \beta^k$.  Partition the unit box $[0,1]^k$ into $\beta^ k$ boxes
  $B_1,B_2,\ldots ,B_{\beta^k}$ as follows: We split up $[0,1]$ into $ \beta $ disjoint consecutive 
 intervals of length $\beta ^{-1}$ each. This gives $\beta^k$ boxes in total.

 Since $|Y|> 2m\cdot \beta^k$, there exists a box $B_j$ such that at least $2m$ points from $Y$ are contained in it.
 By Lemma~\ref{lem:3incube}, applied with $\gamma=\delta=\beta^{-1},  k_1=k$ and $k_2=0$, there exist $p,q\in B_j\cap Y$
 such that $|p-q|^2\leq \frac{1}{2}\left(1+\frac{1}{2m-1}\right)\beta^{-2}k$. We conclude
\begin{align*}
  \alpha^2 k< |p-q|^2\leq \frac{1}{2}\left(1+\frac{1}{2m-1}\right) \beta^{-2}k, \quad \text{implying}
  \quad \alpha< \sqrt{\frac{1}{2}\left(1+\frac{1}{2m-1}\right)}\beta^{-1}.
\end{align*}
However, by the choice of $\beta$, we have
$\alpha< \sqrt{\frac{1}{2}\left(1+\frac{1}{2m-1}\right)}\beta^{-1}\leq \alpha$,
a contradiction. 
\end{proof}

The following lemma is a version of Lemma~\ref{lem:hypercube1} which improves the bound in a certain range of~$\alpha$.

\begin{lemma} \label{lem:hypercube2}
  Let $\sqrt{\frac{100}{1791}}<\alpha < \sqrt{\frac{100}{199}}$ and $Y\subseteq [0,1]^k$
  such that  $|u-v|> \alpha \sqrt{k}$ for every two distinct points $u,v\in Y$. Then,
\begin{align*}
|Y|\leq 600 \cdot 3^{\frac{9}{8}\left(1-\frac{199}{100}\alpha^2\right)k}. 
\end{align*}
\end{lemma}
\begin{proof}
Let $a=\frac{9}{8}(1-\frac{199}{100}\alpha^2)$. Note that $0<a <1$.
Partition the unit box $[0,1]^k$ into $3^{\lceil ak\rceil}$ boxes $B_1,B_2,\ldots ,B_{3^{\lceil ak \rceil}}$ as follows:
Let $I=\{1,2,\ldots, \lceil ak \rceil\}\subseteq [k]$.
For the coordinates in $I$, we split up $[0,1]$ into 3 disjoint consecutive
$[0,1]=[0,\frac{1}{3})\cup [\frac{1}{3},\frac{2}{3}) \cup [\frac{2}{3},1]$ 
 intervals of length $\frac{1}{3}$ each. 
 If $|Y|> 200 \cdot 3^{\lceil ak \rceil}$, then there exists a box $B_j$ such that at least 200 points from $Y$ are contained in it.
 By Lemma~\ref{lem:3incube}, applied with $m=200$, $\delta=\frac{1}{3}$, $\gamma=1$, $k_1=\lceil ak \rceil$ and $k_2=k-k_1$,
 there exist $p,q\in B_j\cap Y$ such that 
\begin{align*}
  \frac{|p-q|^2}{k}\leq \frac{100}{199}\left(\frac{1}{3}\right)^2  \frac{\lceil ak \rceil}{k}+  \frac{100}{199} \frac{k-\lceil ak \rceil}{k}
  \leq \frac{100}{199}-\frac{8}{9}\frac{100}{199}a=\alpha^2, 
\end{align*}
contradicting $\alpha^2 k< |p-q|^2$. We conclude that 
\begin{equation*}
|Y|\leq 200 \cdot 3^{\lceil ak \rceil}\leq 600 \cdot 3^{ ak }=600 \cdot 3^{\frac{9}{8}\left(1-\frac{199}{100}\alpha^2\right)k}. \qedhere
\end{equation*}
\end{proof}

\begin{lemma} \label{lem:hypercube3}
  There exists $k_0$ such that for all integers $k\geq k_0$ the following holds. Let $0< \alpha < 0.99$ and let  $Y\subseteq [0,1]^k$ such that  $|u-v|> \alpha \sqrt{k}$ for every two distinct points $u,v\in Y$.
  Then $|Y|\alpha^k\leq 0.999^k$.
\end{lemma}
\begin{proof}
Let $k_0$ be sufficiently large for the following proof to hold. First, assume $\sqrt{\frac{100}{199}}<\alpha <0.99$. Then $|Y|\leq 200$ by Lemma~\ref{lem:sym},
  applied with $m=200$. Thus, 
\begin{align*}
|Y|\alpha^k \leq 200 \alpha^k \leq 0.999^k.
\end{align*}
Next, assume $0.29\leq \alpha\leq \sqrt{\frac{100}{199}}$. Then by Lemma~\ref{lem:hypercube2},
\begin{align*}
|Y|\alpha^k \leq 600  \cdot \left(3^{\frac{9}{8}\left(1-\frac{199}{100}\alpha^2\right)} \alpha\right)^k \leq 0.999^k.
\end{align*}
Finally, assume $0<\alpha \leq 0.29$. Then by Lemma~\ref{lem:hypercube1}, applied with $m=100$, 
\begin{equation*}
|Y|\alpha^k \leq 200 \left( \left\lceil  \frac{\sqrt{\frac{1}{2}\left(1+\frac{1}{199}\right)}}{\alpha}\right\rceil\alpha \right)^k
\leq 
200 \left( \sqrt{\frac{100}{199}}+\alpha \right)^k\leq 0.999^k.  \qedhere
\end{equation*}
\end{proof}

\subsection{Proofs of Theorems~\ref{thm:asymp-st}, ~\ref{thm:asymp-hc+}, and \ref{thm:fixedn}}

First, we quickly demonstrate how Lemma~\ref{lem:sym} implies Theorem~\ref{thm:fixedn}. 

\paragraph{Proof of Theorem~\ref{thm:fixedn}.} 
Let $X\subseteq [0,1]^k$ be a point set of size $n$. We run Algorithm 1 from Section~\ref{pathgreedy}.
Let $F_i$ be the collection of paths at the $i$-th step, let $e_i$ be the edge added in the $i$-th step,
and let $F=F_{n-1}$ be the final Hamiltonian path.

We claim that $|e_i| \leq \sqrt{\frac{2}{3}k}$ for $i\leq n-2$. Let $e_i=xy$. The vertices $x$ and $y$
are endpoints of two different paths in $F_{i-1}$. Since $F_{i-1}$ has at least $n-(i-1)\geq n-(n-2-1)=3$
components, there exists a component containing neither $x$, nor $y$. Let $z\in X$ be an endpoint of
the path forming this component. Since $e_i=xy$ was chosen in step $i$,
but $xz$ and $yz$ were not, we have $|xy| \leq |xz|$ and $|xy| \leq |yz|$.
By applying Lemma~\ref{lem:sym} to the set $\{x,y,z\}$, we get that $|e_i|=|xy| \leq \sqrt{\frac{2}{3}k}$.
Note that $|e_{n-1}| \leq \sqrt{k}$  trivially.

Now, let $f=ab$ be the edge where $a$ and $b$ are the  two endpoints of the final path $F$. 
Set $H=F+f$ to be the Hamiltonian cycle when $f$ is added to $F$. Since $|f| \leq \sqrt{k}$  trivially,
we get 
\begin{align*}
  S_k(H)&=\sum_{e\in H}|e|^k= |f|^k+|e_{n-1}|^k+ \sum_{i=1}^{n-2} |e_i|^k \leq
  2\left(\sqrt{k}\right)^k+ (n-2)\left(\sqrt{\frac{2}{3}k}\right)^k.
\end{align*}

Consequently,
\[ \pushQED{\qed} 
s_k^{\texttt{HC}}(n)\leq s_k(H) \leq 2^{1/k}\sqrt{k} \ (1+o_k(1)). \qedhere \popQED \]

\paragraph{Proof of Theorem~\ref{thm:asymp-st}.}
Let $k$ be sufficiently large and let $X\subseteq [0,1]^k$ be a finite point set. Set
\begin{align*}
  \ell:=\left\lceil \log_{1+\frac{1}{k}}\left(0.9 k^{\frac{3}{4}}\right) \right\rceil=
  O(k \log k) \quad \quad \text{and} \quad \quad a_i:=\frac{(1+\frac{1}{k})^i}{k^{\frac{3}{4}}}
\end{align*}
for integers $i$, $0\leq i\leq \ell$. Note that 
\begin{align*}
  \frac{a_{i+1}}{a_i}=1+\frac{1}{k}\ \text{for} \ i\in\{0,1,\dots,\ell-1\}, \quad \text{and}
  \quad a_0 < a_1 < a_2 < \cdots < a_{\ell-1}\leq 0.9 \leq a_\ell. 
\end{align*}
Construct a minimum spanning tree $T$ on vertex set $X$ by successively
joining points from $X$ at minimal distance from each other, given the new edge does not create a cycle.
For $0\leq i\leq \ell$, let $F_i$ be the forest with vertex set $X$ and edges $e\in T$ such that $|e|\leq a_i\sqrt{k}$.
Then, $F_0\subseteq F_1 \subseteq \dots \subseteq F_{\ell} \subseteq T$ since the sequence $(a_i)$ is increasing.
If $x,y\in X$ are in different components of $F_i$, then $|x-y|> a_i \sqrt{k}$.

We have $a_0= k^{-3/4}$.
For an edge $e=xy\in F_0$, let $B_e$ be the open ball of radius $|e|/4$ and center $\frac{1}{2}(x+y)$.
Since $F_0\subseteq T$, by Lemma~\ref{lem:bm92}, the balls $B_e$, $e\in F_0$ are disjoint.
Also, $|e|\leq a_0 \sqrt{k}=k^{-1/4}$. Denote by $V_k$ for the volume of the $k$-dimensional unit ball.
It is well-known that
\begin{equation*} \label{eq:V_k}
V_k= \begin{cases}
\dfrac{\pi^{k/2}}{(k/2)!} & {\rm if \ } k \ {\rm is \ even},
\medskip \\
\dfrac{2^k \cdot \pi^{(k-1)/2} \, ((k-1)/2)!}{k!} & {\rm if \ }
k \ {\rm is \ odd}.
\end{cases}
\end{equation*}

By Stirling's approximation,
$V_k \sim \frac{1}{\sqrt{k\pi}} (\frac{2\pi e}{k})^{k/2}$. Since  $\bigcup_{e\in F_0} B_e\subseteq [-k^{-1/4},1+k^{-1/4}]$,
we have 
\begin{gather*}
\sum_{e\in F_0}\left(\frac{|e|}{4}\right)^k V_k \leq ((1+2k^{-1/4}))^k, \quad
\text{and thus} \quad 
\sum_{e\in F_0}|e|^k\leq \frac{4^k(1+2k^{-1/4})^k}{V_k} 
\leq  (0.97\sqrt{k})^k,
\end{gather*}
for $k$ sufficiently large. Now, let $i\in \{0,1,\ldots,\ell-1\}$.
Let $Y \subseteq X$ be a set of vertices containing exactly one vertex
from every component of $F_i$. Then $|y-y'|> a_i \sqrt{k}$ for every pair $y\neq y'\in Y$,
and  $|F_{i+1}\setminus F_i|\leq |Y|-1$.
By Lemma~\ref{lem:hypercube3} we have $a_i^k|Y| \leq  0.999^k$ for $i\leq \ell$. Thus,
\begin{gather*}
  \sum_{e\in F_{i+1}\setminus F_i} |e|^k\leq (a_{i+1}\sqrt{k})^k|Y|= (a_{i}\sqrt{k})^k|Y|\left(1+\frac{1}{k}\right)^k
  \leq  3 \cdot (0.999 \sqrt{k})^k,
\end{gather*}
for $i\leq \ell$.
Therefore, 
\begin{align*}
  \sum_{e\in F_\ell} |e|^k&= \sum_{e\in F_0} |e|^k +\sum_{i=0}^{\ell-1} \sum_{e\in F_{i+1}\setminus F_i} |e|^k
  \leq \left(0.97\sqrt{k}\right)^k+ 3 \ell \cdot (0.999 \sqrt{k})^k,
\end{align*}
for $k$ sufficiently large. If the forest $F_\ell$ consist of at least three components then three points $p,q,r\in X$,
from different components each, have pairwise distance at least $0.9\sqrt{k}\geq \sqrt{\frac{2}{3}k}$.
This contradicts Lemma~\ref{lem:sym}. Therefore, $F_\ell$ has at most 2 components and thus there is at most
one edge $f$ in $T$ which is not in $F_\ell$.
We  conclude
 \begin{align*}
    \sum_{e\in T} |e|^k=\sum_{e\in F_\ell} |e|^k+ |f|^k\leq   \left(\sqrt{k}\right)^k(1+o_k(k^{-1})),
\end{align*}
 which implies that for the family of minimum spanning trees, we have $s_k^{\texttt{ST}} \leq  \sqrt{k} \ (1+o_k(1))$,
 completing the proof of Theorem~\ref{thm:asymp-st}.  
 \qed
 
 \medskip
 We remark that by applying Lemma~\ref{lem:bm92b} to $T$, there exists a Hamiltonian cycle $H$ on
 vertex set $X$ satisfying 
 \begin{align*}
    \sum_{e\in H} |e|^k\leq \frac{2}{3}\cdot 3^k
    \sum_{e\in T} |e|^k \leq   \left(3\sqrt{k}\right)^k(1+o_k(k^{-1})),
\end{align*}
 implying  that for the family of Hamiltonian cycles, we have $s_k^{\texttt{HC}} \leq  3\sqrt{k} \ (1+o_k(1))$.

\paragraph{Proof of Theorem~\ref{thm:asymp-hc+}.}
Create a forest $F$ by successively joining points from $X$ at minimal distance from each other,
given the new edge $e$ does not create a cycle and satisfies $|e|\leq k^{-1/4}$.
This process stops when there is no such edge left. Let the trees $T_1,\ldots, T_N$ be the components of~$F$.   
Every two vertices from different $T_i$'s have pairwise distance at least $k^{-1/4}$.

For an edge
$e=xy\in F$, let $B_e$ be the open ball of radius $|e|/4$ and
center $\frac{1}{2}(x+y)$. By Lemma~\ref{lem:bm92}, the balls $B_e$, $e\in F$ are disjoint.
Also, $|e|\leq k^{-\frac{1}{4}}$. We have $\bigcup_{e\in F} B_e\subseteq [-k^{-1/4},1+k^{-1/4}]$.
Writing $V_k$ for the volume of the $k$-dimensional unit ball,
we have 
\begin{gather*}
\sum_{e\in F}\left(\frac{|e|}{4}\right)^k V_k \leq ((1+2k^{-1/4}))^k \quad
\text{and thus} \quad 
\sum_{e\in F}|e|^k\leq \frac{4^k(1+2k^{-1/4})^k}{V_k} 
\leq  (0.97\sqrt{k})^k, 
\end{gather*}
for $k$ sufficiently large. Since the trees $T_1,\ldots, T_N$ decompose the edge set of the forest $F$, we have
\begin{align}
\label{Thm16_trees}
\sum_{i=1}^N \sum_{e\in T_i} |e|^k= \sum_{e\in F}|e|^k\leq  (0.97\sqrt{k})^k.
\end{align}
By Lemma~\ref{lem:bm92b}, for each $i\in [N]$, there exists a Hamiltonian cycle $H_i$ on $V(T_i)$ such that 
\begin{align}
\label{Thm16_ham}
\sum_{e\in H_i} |e|^k\leq 3^k \sum_{e\in T_i}|e|^k.
\end{align}
Let $F_0$ be the collection of paths obtained by taking the union of all $H_i$, and removing an edge from each cycle. Then, by using \eqref{Thm16_trees} and \eqref{Thm16_ham}, we obtain
\begin{align}
\label{Thm16_start}
\sum_{e\in F_0} |e|^k \leq \sum_{i=1}^N\sum_{e\in H_i} |e|^k\leq 3^k \sum_{i=1}^N\sum_{e\in T_i}|e|^k  \leq (2.91\sqrt{k})^k.
\end{align}
Now, run Algorithm 1 from Section~\ref{pathgreedy} initialized with $F_0$ (instead of the empty graph). Recall that this algorithm adds edges of minimum weight such that in each step we maintain a collection of paths. Denote by $Q$ the final path which is created by this algorithm.  Set
\begin{align*}
  \ell:=\left\lceil \log_{1+\frac{1}{k}}\left(0.9 k^{\frac{3}{4}}\right) \right\rceil=
  O(k \log k) \quad \quad \text{and} \quad \quad a_i:=\frac{(1+\frac{1}{k})^i}{k^{\frac{3}{4}}}
\end{align*}
for integers $i$, $0\leq i\leq \ell$. For $0\leq i\leq \ell$, let $F_i$ be the collection of paths with vertex set $X$ and edges $e\in Q$ such that $|e|\leq a_i\sqrt{k}$.
Then, $F_0\subseteq F_1 \subseteq \dots \subseteq F_{\ell} \subseteq Q$ since the sequence $(a_i)$ is increasing. If $x,y\in X$ are in different components of $F_i$, then $|x-y|> a_i \sqrt{k}$. Now, let $i\in \{0,1,\ldots,\ell-1\}$.
Let $Y \subseteq X$ be a set of vertices containing exactly one endpoint of each path of $F_i$. Then $|y-y'|> a_i \sqrt{k}$ for every pair $y\neq y'\in Y$,
and  $|F_{i+1}\setminus F_i|\leq |Y|-1$.
By Lemma~\ref{lem:hypercube3} we have $a_i^k|Y| \leq  0.999^k\leq 1$ for $i\leq \ell$. Thus,
\begin{gather}
\label{Thm16_analysis}
  \sum_{e\in F_{i+1}\setminus F_i} |e|^k\leq (a_{i+1}\sqrt{k})^k|Y|= (a_{i}\sqrt{k})^k|Y|\left(1+\frac{1}{k}\right)^k
  \leq  3 \cdot \sqrt{k}^k,
\end{gather}
for $i\leq \ell$.
Therefore, by combining \eqref{Thm16_start} with \eqref{Thm16_analysis}, we obtain
\begin{align}
\label{Thm16_Fl}
  \sum_{e\in F_\ell} |e|^k&= \sum_{e\in F_0} |e|^k +\sum_{i=0}^{\ell-1} \sum_{e\in F_{i+1}\setminus F_i} |e|^k
  \leq \left(2.91\sqrt{k}\right)^k+ 3 \ell \cdot  \sqrt{k}^k,
\end{align}
for $k$ sufficiently large. Similarly, as in the proof of Theorem~\ref{thm:asymp-st}, $F_\ell$ has at most $2$ components. Thus, using \eqref{Thm16_Fl}, the path $Q$ satisfies
\begin{gather*}
\sum_{e\in Q}|e|^k \leq \sum_{e\in F_\ell}|e|^k +\sqrt{k}^k \leq  (2.91\sqrt{k})^k (1+o_k(1)).
\end{gather*}
Adding one final edge $f$ of weight at most $|f| \leq \sqrt{k}$ to $Q$ we obtain a Hamiltonian cycle with the desired properties.
\qed

\section{Concluding Remarks} \label{sec:remarks}

The upper bounds we obtained on the lengths of Hamiltonian cycles have the following implications
for the existence of perfect matchings whose cost is bounded from above by a constant (depending on $k$). 
For example, Theorems~\ref{thm:s} and~\ref{thm:vertices} have the following implications.
The proofs of Corollary~\ref{cor:matching} and that of Corollary~\ref{cor:matching-vertices}
are analogous to the proof of Corollary~\ref{cor:matching-square} below. 

\begin{corollary} \label{cor:matching}
 Given $n$ points in $[0,1]^k$, where $k \geq 3$, and $n$ is even, 
 there exists a perfect matching $M$ of the $n$ points such that
  $\left(\sum_{e \in M} |e|^k \right)^{1/k} \leq 3 \sqrt5 \left(\frac13 \right)^{1/k} \cdot \sqrt{k}$.
The matching $M$ can be computed in time proportional to that needed for computing a MST of the points,
  in particular, in subquadratic time. 
\end{corollary}

\begin{corollary} \label{cor:matching-vertices}
There exists an integer $k_0$ such that for all $k\geq k_0$ the following holds.
If $X$ is any even-size subset of vertices of $\{0,1\}^k$, then there exists a perfect matching
$M$ of $X$ such that $s_k(M) \leq \sqrt{k}$. This bound is best possible.  
\end{corollary}

Recall that a MST of $n$ points in $\RR^k$ (with respect to Euclidean distances) can be computed in
$O\left ( n^{2 - \frac{2}{\lceil k/2 \rceil +1} + \eps} \right)$ time, for any $\eps>0$~\cite{AES91}.
We also deduce the following related results (formulated here for the planar case, $k=2$.)

\begin{corollary} \label{cor:nearest} 
  Let $x_1,\ldots,x_n$ be $n \geq 2$ points in the unit square. Let $d_i$ be the distance between $x_i$
  and its nearest point (other than $x_i$). Then the following inequality holds:  $\sum_{i=1}^n d_i^2 \leq 4$.
\end{corollary}
\begin{proof}
  Consider a Hamiltonian cycle, say $x_1,\ldots,x_n$, whose cost $S_2(H)$ is at most $4$.
  The distance from $x_i$ to its nearest point is at most $|x_i - x_{i+1}|$, for $i=1,\ldots,n$.
  By squaring the $n$ inequalities and adding them up, the claimed inequality follows.
\end{proof}

An alternative proof of Corollary~\ref{cor:nearest} can be found in~\cite[Problem G.27]{Sze96}.

\begin{corollary} \label{cor:matching-square} 
Let $x_1,\ldots,x_n$ be $n \geq 2$ points in the unit square, where $n$ is even.
Then there exists a perfect matching $M$ such that  $\sum_{e \in M} |e|^2 \leq 2$.
This bound is the best possible. 
\end{corollary}
\begin{proof}
  Consider a Hamiltonian cycle, say $H=x_1,\ldots,x_n$, whose cost $S_2(H)$ is at most $4$.
  $H$ can be decomposed into two perfect matchings, one of which has a cost at most $2$,
  as required.
  
  The lower bounds for $n=2$ and $n=4$ are immediate (see Figure~\ref{fig:square}).
  For every even $n \geq 6$ and $\eps>0$, there are $n$ points (in the neighborhoods of the
  four corners of the square) such that $\sum_{e \in M} |e|^2 \geq 2-\eps$.
\end{proof}

\smallskip
We have improved the upper bound of Bollob{\'a}s and Meir~\cite{BM92} by more than 25 percent in
the exact formulation and by more than 67 percent in the asymptotic formulation.
Apart from some doubt concerning the values of $s_3^{\texttt{HC}}(n)$ and $s_4^{\texttt{HC}}(n)$,
we think that their lower bound gives the right answer for every higher dimension.
In view of Theorem~\ref{thm:s}\,(i) we adjust their conjecture as follows: 

\begin{conjecture} \label{conj:cycle}
  For Hamiltonian cycles, the following equalities hold:
 \[ s_k^{\texttt{HC}}= \begin{cases}
    2^{7/6}, & \text{ for } k =3,  \\ 
    2^{1/k} \cdot \sqrt{k}, & \text{ for } k \geq 4.
  \end{cases}
  \]
\end{conjecture}
\paragraph{Hamiltonian path.}
If one was looking for a Hamiltonian path, instead of a Hamiltonian cycle, then the $2$-point extremal lower bound
example (given by a cube diagonal) loses a factor of $2$ (or with scaling $2^{1/k}$); and so the question arises:
is it still the best example, or maybe only for large $k$? Analogous to the situation for Hamiltonian cycles,
we think that there is a threshold value for $k$ after which the extremal  examples stabilizes at the $2$-point example.
The  threshold values for cycles and paths seem to differ, see Conjecture~\ref{conj:path} below.

The current upper bound proofs essentially remain the same as for Hamiltonian cycles, with the change
that the last edge is not needed.
Some upper bounds remain unchanged, and others do improve.
In particular, $s_2^{\texttt{HP}} \leq s_2^{\texttt{HC}} = 2$ remains unchanged, whereas $s_2^{\texttt{HP}} \geq \sqrt3$
is implied by the two extremal examples in Figure~\ref{fig:square}\,(left and right).

From the other direction, for small values of $k$ consider once again a binary code of length $k$ with minimum
distance $2$ given by the set of all $x \in \{0,1\}^k$ with an even number of $1$'s.
It yields the values specified below.

\begin{conjecture} \label{conj:path}
  For Hamiltonian paths, the following equalities hold:
  \[ s_k^{\texttt{HP}}= \begin{cases}
    \sqrt{3}, & \text{ for } k =2, \\ 
    \left( 2^{k-1} -1 \right)^{1/k} \cdot \sqrt2, & \text{ for } k = 3,4,5,6, \\
    \sqrt{k}, & \text{ for } k \geq 7. 
  \end{cases}
  \]
\end{conjecture}

\paragraph{Further improvement.}
One might wonder where the next possible improvement is? We feel that it is in Lemma~\ref{lem:sek}:
It states that there is a Hamiltonian cycle such that each edge of the cycle is using at most 3 tree edges,
yet the average usage is slightly less than 2. If it was true that every tree edge is used at most twice,
then we would get a 2/3 factor improvement in the upper bound. However, the example of a tree with edges
$ab,bc,cd,de,cf,fg$ shows that this is not the case. Still, it is likely that there is a way to gain more 
in a tree to cycle or path conversion.

\paragraph{A different version.}
We conclude with yet another version of the problem. 
Instead of the unit cube $[0,1]^k \subset \RR^k$, let the diameter of the point set be at most $1$: 
That is, $\diam(X) \leq 1$, where $X \subset \RR^k$ and $|X|=n$. 
What are the extremal values of the (say, unscaled) costs of a shortest Hamiltonian cycle (and path)
for $n$ points in $\RR^k$ under this constraint? Are they given by the vertices of a unit simplex in $\RR^k$
($k+1$ and $k$, respectively)?

\section*{Acknowledgements}
We thank anonymous referees for carefully reading the manuscript.

\section*{Conflict of interest}
The authors have no relevant financial or non-financial interests to disclose.

\section{Appendix}

\subsection*{Container shapes with a tight bound when $k=2$}

A key fact in deriving the tight bound, when $k=2$, for the cycle of $n$ points in the unit square is the following
tight bound in a right triangle~\cite{New82}; see also~\cite{Bol06}. 

\begin{lemma} {\rm \cite{New82}}  \label{lem:newman}
  Let $X$ be a set of $n \geq 2$ points in a right triangle $\Delta$ whose sides are $a \leq b \leq c$. Then there is an
  extended path connecting the endpoints of $c$ that visits all points in $X$ and for which
  $\sum |e|^2 \leq c^2$. In particular, $X$ admits a Hamiltonian path $P$ for which
  $\sum_{e \in P} |e|^2 \leq c^2$. This bound is the best possible. 
\end{lemma}

This result relies on a repeated application of the following simple corollary of the Cosine Law.
It allows one to make shortcuts in a path or cycle at vertices where the two adjacent edges make an acute angle. 

\begin{lemma} {\rm \cite{New82}}  \label{lem:shortcut}
  Let $\Delta$ be an triangle whose sides are $a \leq b \leq c$, and let $\gamma$ be the
  interior angle opposite to $c$. If $\gamma \leq 90^\circ$, then $c^2 \leq a^2 + b^2$.
\end{lemma}

We now exhibit two other container shapes for which we can deduce a tight bound.
Lemma~\ref{lem:non-obtuse-hp} below is an extension of Lemma~\ref{lem:newman}.

\begin{figure}[ht]
\centering   
\includegraphics[scale=1.1]{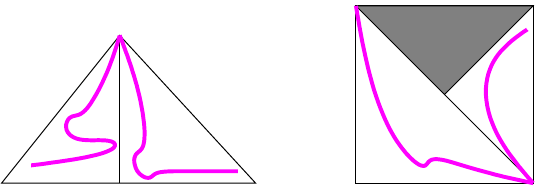}      
\caption{Left: Lemma~\ref{lem:non-obtuse-hp}. Right: Lemma~\ref{lem:envelope}.}
\label{fig:shapes}
\end{figure}

\begin{lemma} \label{lem:non-obtuse-hp}
  Let $X$ be a set of $n \geq 2$ points in a non-obtuse triangle $\Delta$ whose sides are $a \leq b \leq c$.
  Then there is an extended path connecting the endpoints of $c$ that visits all points in $\Delta$ and for which
  $\sum |e|^2 \leq a^2 + b^2$. In particular, $X$ admits a Hamiltonian path $P$ for which
  $\sum_{e \in P} |e|^2 \leq a^2 + b^2$. This bound is the best possible.
\end{lemma}
\begin{proof}
  Let the altitude corresponding to $c$ divide $\Delta$ into two right triangles.
  Consider the path obtained by concatenating the extended paths for the two right triangles.
  Further shortcut the path at the concatenation vertex by using Lemma~\ref{lem:shortcut} to obtain 
  a Hamiltonian path $P$ for which $\sum_{e \in P} |e|^2 \leq a^2 + b^2$, see Figure~\ref{fig:shapes}\,(left).
  The three vertices of $\Delta$ provide a tight example.
\end{proof}

\begin{lemma} \label{lem:non-obtuse-hc}
  Let $X$ be a set of $n \geq 2$ points in a non-obtuse triangle $\Delta$ whose sides are $a \leq b \leq c$.
  Then $X$ admits a Hamiltonian cycle $H$ for which $\sum_{e \in H} |e|^2 \leq a^2 + b^2+ c^2$. This bound is
  the best possible.
\end{lemma}
\begin{proof}
  By Lemma~\ref{lem:non-obtuse-hp}, $X$ admits a Hamiltonian path $P$ for which $\sum_{e \in P} |e|^2 \leq a^2 + b^2$.
  Connecting the endpoints of this path (via an edge of length at most $c$) yields a Hamiltonian cycle $H$
  for which $\sum_{e \in H} |e|^2 \leq a^2 + b^2+ c^2$.
\end{proof}

By Lemma~\ref{lem:non-obtuse-hp} and ~\ref{lem:non-obtuse-hc}, we obtain the following corollary.
\begin{corollary} \label{cor:non-obtuse}
  Let $X$ be a finite point set in in a non-obtuse triangle $\Delta$ whose sides are $a \leq b \leq c \leq 1$.
  Then
\begin{enumerate} \itemsep 1pt
  \item $X$ admits a Hamiltonian path $P$ for which $\sum_{e \in P} |e|^2 \leq 2$.
  \item $X$ admits a Hamiltonian cycle $H$ for which $\sum_{e \in H} |e|^2 \leq 3$.
\end{enumerate}
\end{corollary}

\begin{lemma} \label{lem:envelope}
    Let $U$ be a unit square centered at $o$ and let $ab$ be one of its four sides.
    Let $X$ be a set of $n \geq 2$ points in $V:=U \setminus \Delta{oab}$ ($V$ as a closed set).
    Then $X$ admits a Hamiltonian path $P$ for which $\sum_{e \in P} |e|^2 \leq 3$.
    This bound is the best possible. 
\end{lemma}
\begin{proof}
  Subdivide $V$ into two right triangles as shown in Figure~\ref{fig:shapes}\,(right).
    Consider the path obtained by concatenating the extended paths for the two right triangles.
    Further shortcut the path by using Lemma~\ref{lem:shortcut} to obtain a Hamiltonian path $P$ for which
    $\sum_{e \in P} |e|^2 \leq 1^2 + (\sqrt2)^2 =3$. 
  The $4$- and $5$-point examples in Figure~\ref{fig:square} show that this bound is tight.
\end{proof}

\subsection*{A different version of Theorem~\ref{thm:fixedn}}

We remark that the proof of Theorem~\ref{thm:fixedn} can be extended for point sets of size $n$,
when $n$ is slowly growing in $k$, to obtain an upper bound sharper than that in Theorem~\ref{thm:asymp-hc+}.

\begin{theorem} \label{thm:smalln}
  The following bounds are in effect:
  \begin{itemize} \itemsep 1pt
  \item [{\rm (i)}] If $n \leq 2^k +2$, then there exists a Hamiltonian cycle $H$ such that
   \[  S_k(H) \leq \left(2 + \left(\frac83\right)^{k/2} \right) k^{k/2}. \]  
Consequently, $s_k^{\texttt{HC}}(n)  \leq 1.64 \sqrt{k}$, for $k$ sufficiently large. 
\item [{\rm (ii)}]
 If $n \leq 2^k$, then there exists a Hamiltonian cycle $H$ such that
  \[  S_k(H) \leq \left(200 + 2.01^{k/2} \right) k^{k/2}. \]  
Consequently, $s_k^{\texttt{HC}}(n)  \leq 1.42 \sqrt{k}$, for $k$ sufficiently large. 
\end{itemize}
\end{theorem}
\begin{proof}
  Let $X\subseteq [0,1]^k$ be a point set of size $n$.
  We run Algorithm 1 from Section~\ref{pathgreedy}. Let $F_i$ be the collection of paths at the $i$-th step,
  let $e_i$ be the edge added in the $i$-th step, and let $F=F_{n-1}$ be the final Hamiltonian path.

(i) We know that $|e_i| \leq \sqrt{\frac{2}{3}k}$ for $i\leq n-2$. Note that $|e_{n-1}| \leq \sqrt{k}$  trivially.
Now, let $f=ab$ be the edge where $a$ and $b$ are the  two endpoints of the final path $F$. 
Set $H=F+f$ to be the Hamiltonian cycle when $f$ is added to $F$. Since $|f| \leq \sqrt{k}$  trivially,
we get 
\begin{align*}
  S_k(H)&=\sum_{e\in H}|e|^k= |f|^k+|e_{n-1}|^k+ \sum_{i=1}^{n-2} |e_i|^k \leq
  2\left(\sqrt{k}\right)^k+ (n-2)\left(\sqrt{\frac{2}{3}k}\right)^k \\
  &\leq \left( 2 + 2^k \cdot \left(\frac{2}{3}\right)^{k/2} \right) \cdot k^{k/2}
    = \left(2 + \left(\frac83\right)^{k/2} \right) k^{k/2}. 
\end{align*}

Consequently, $s_k^{\texttt{HC}}(n) \leq s_k(H) \leq 1.64 \sqrt{k}$, for $k$ sufficiently large. 

\medskip
(ii) We classify the edges $e\in F$ into two types.
  \begin{enumerate} \itemsep 1pt
    \item \emph{short} edges: $|e|^2 \leq \frac{100k}{199}$.
  \item \emph{long} edges: $\frac{100k}{199} < |e|^2 $.
  \end{enumerate}
 The number of short edges $e\in F$ is at most $n \leq 2^k$, trivially.
 The number of long edges $e\in F$ is at most $199$ by Lemma~\ref{lem:sym} applied with $m=200$.
Now, let $f=ab$ be the edge where $a$ and $b$ are the  two endpoints of the final path $F$. 
Set $H=F+f$ to be the Hamiltonian cycle when $f$ is added to $F$. Since $|f| \leq \sqrt{k}$  trivially,
we get 
\begin{align*}
  S_k(H) &\leq \left( 200 + 2^k \cdot \left(\frac{100}{199}\right)^{k/2} \right) \cdot k^{k/2}
    \leq \left(200 + 2.01^{k/2} \right) k^{k/2}. 
\end{align*}
Consequently, $s_k^{\texttt{HC}}(n) \leq s_k(H) \leq 1.42 \sqrt{k}$, for $k$ sufficiently large. 
\end{proof}

\end{document}